\documentclass[11pt, a4paper, reqno, twoside]{article}

\usepackage{amssymb,ifthen,amsfonts,amsmath,theorem,graphicx}

\newtheorem{thm}{Theorem}[section]

\newtheorem{Remark}[thm]{Remark}

\newcommand{\qed}{\hfill \ensuremath{\Box}}

\newenvironment{proof}{\vspace{1ex}\noindent{\it Proof.}\hspace{0.5em}}
	{\hfill\qed\vspace{1ex}}

\newcommand{\R}{\mathbb R}
\def\RR{\mathop{\trm I\kern-,057cm R}}
\def\medint{-\kern-,375cm\int}

\def\div{\mathop{\rm div}}

\def\[{\bigl(}
\def\]{\bigl)}

\def\rz{\ifmmode{I\hskip -3.2pt R}
        \else{\hbox{$I\hskip -3.2pt R$}}\fi} 

\begin{document}

\title{Overdetermined boundary value problems for the $\infty$-Laplacian}

\author{G.~Buttazzo \& B.~Kawohl}


\maketitle


\bigskip\bigskip

\noindent
{\bf\small Abstract:} We consider overdetermined boundary value problems for the $\infty$-Laplacian in a domain $\Omega$ of $\R^n$ and discuss what kind of implications on the geometry of $\Omega$ the existence of a solution may have. The classical $\infty$-Laplacian, the normalized or game-theoretic $\infty$-Laplacian and the limit of the $p$-Laplacian as $p\to \infty$ are considered and provide different answers, even if we restrict our domains to those that have only web-functions as solutions.



\section{Motivation }\label{section1}\setcounter{equation}{0}

Suppose that $\Omega\subset\R^n$ is connected and bounded, with boundary at least of class $C^1$, and that $u\in C^1(\overline{\Omega})$ is a positive solution of the overdetermined boundary value problem
\begin{eqnarray}\label{op}
-\Delta_p u_p:=-\div\left(|\nabla u_p|^{p-2}\nabla u_p\right)&=&1\quad\hbox{in }\Omega,\\
\label{bc1}
u_p&=&0\quad\hbox{on }\partial\Omega,\\
\label{bc2}
-\frac{\partial u_p}{\partial\nu}&=&a\quad\hbox{on }\partial\Omega,
\end{eqnarray}
where $p\in(1,\infty)$ and $a$ is a positive constant. Does this have consequences on the geometry of $\Omega$? This question was answered in 1971 for $p=2$ by Serrin \cite{S} and Weinberger \cite{W}, and for general $p$ in 1987 by Garofalo and Lewis \cite{GL}. See also Farina and Kawohl \cite{FK} for related results. In both cases the domain $\Omega$ must be a ball of fixed radius related to $a$. This result leads us to the question: what happens if the $p$-Laplacian is replaced by the infinity Laplacian?

The answer depends on how we define the $\infty$-Laplacian and the notion of solution. In case of equation \eqref{op} and $p=2$ Serrin and Weinberger had classical $C^2(\Omega)$ solutions in mind, while for general $p\in(1,\infty)$ the solutions were weak in the sense that
$$\int_\Omega|\nabla u_p|^{p-2}\nabla u_p\nabla v\ dx=\int_\Omega v\ dx \quad\hbox{ for every }v\in W^{1,p}_0(\Omega).$$

\section{The classical $\infty$-Laplacian }\label{section2}\setcounter{equation}{0}

The classical $\infty$-Laplacian operator is usually defined as $\Delta_\infty u:= $ $\langle D^2uDu,$ $Du\rangle$, with $Du$ denoting the gradient and $D^2 u$ the Hessian matrix of $u$. For functions in $C^2$ the second directional derivative in direction $\nu$ is given by $\langle D^2u\ \nu,\nu\rangle$. If $\nu$ denotes the direction $-Du/|Du|$ of steepest descent of $u$, the equation $-\Delta_\infty u=1$ can be rewritten as 
\begin{equation}\label{firstdef}
-u_{\nu\nu}|u_\nu|^2=1,
\end{equation}
and if $\Omega$ should happen to be a ball of radius $R$ centered at zero, $u(x)$ is necessarily a radial function. In fact, then
$$u(r)=\frac{3^{4/3}}{4}(R^{4/3}-r^{4/3})\qquad\hbox{and }u_r(R)=-(3R)^{1/3}$$ imply that $R$ must be equal to $a^3/3$ to match both boundary conditions. Notice that this function is exactly of class $C^{1,1/4}$, which is the conjectured optimal regularity for $\infty$-harmonic functions $v$, that is for functions satisfying $\Delta_\infty v=0$.

Therefore we cannot expect classical solutions. Since the equation is not in divergence form, we cannot expect a notion of weak solution either. Instead we define a {\it viscosity solution} $u$ of the equation
$$F(Du,D^2u):=-\langle D^2uDu,Du\rangle-1=0$$
as a continuous function which is both a viscosity sub- and viscosity supersolution. A viscosity subsolution has the property that $F(D\varphi,D^2\varphi)(x)\le0$ whenever $\varphi$ is a $C^2$-function such that $\varphi-u$ has a local minimum at $x$. A viscosity supersolution has the property that $F(D\psi,D^2\psi)(x)\ge0$ whenever $\psi$ is a $C^2$-function such that $\psi-u$ has a local maximum at $x$, see for instance \cite{CIL}. In our autonomous case we may also assume that $\varphi$ touches $u$ from above at $x$ if we check the definition of subsolutions, and that $\psi$ touches from below at $x$ if we check supersolutions.

Let us see that the explicit radial function $c-kr^{4/3}$, with $k=3^{4/3}/4$ is a viscosity solution of $F(Du,D^2u)=0$ at $x=0$. If $\varphi$ is a smooth function touching $u$ from above, then $\nabla\varphi(0)=0$, so $\varphi_\nu=0$ and $F(D\varphi,D^2\varphi)=-1$, which is less or equal to zero, as required for subsolutions. For supersolutions the set of test functions $\psi$ that touch $u$ from below in the origin is empty, so that the condition for a supersolution is trivially satified. Effects like this happen quite often when viscosity solutions are not smooth. Checking the property of sub- or supersolution is somehow easier in points where the solutions loose smoothness.

Now suppose that $\Omega$ is not necessarily a ball, but a more general smooth domain.

\begin{Remark}{\rm
From every point $x_0$ on $\partial\Omega$ we can follow the line of steepest ascent, parametrized as $x(t)$  by solving the initial value problem 
\begin{equation}\label{ODE}
x(0)=x_0, \qquad\frac{d x_i}{d t}=u_{x_i}\quad\hbox{\rm for small but positive }t.
\end{equation}
A simple calculation shows, assuming that $u$ is locally of class $C^2$,
\begin{equation}
\frac{d }{d t}\left(\left| \frac{d x}{d t}\right|^2\right)=2u_{x_i}u_{x_i x_j}u_{x_j}=-2,
\end{equation} 
so that upon integration from 0 to $t$
\begin{equation}
\left| \frac{d x}{d t}\right|^2=|\nabla u(x(t)|^2=a^2-2t.
\end{equation}
Note that this works until $t$ reaches $a^2/2$, at which time $\nabla u=0$. Subsequently we get the estimate
\begin{equation}
|x(t)-x_0|=\left|\int_0^t x_t(s)\ ds\right|\leq\frac{1}{3}\left( a^3-(a^2-2t)^{3/2}\right)\leq \frac{a^3}{3}.
\end{equation}
This shows that our trajectories can never reach a distance greater than $a^3/3$ from the boundary of $\Omega$ and that any critical point of $u$ that can be approached this way has at most distance $a^3/3$ from $\partial\Omega$.}
\end{Remark}

Notice that the radial solution on a ball is a web-function in the sense of \cite{CFG}, i.e. a function, whose value depends only on the distance to $\partial\Omega$. From now on we assume that a solution of  \eqref{firstdef} \eqref{bc1} \eqref{bc2} happens to be a webfunction for a general domain as well. This may be justified via the Cauchy-Kowalewski Theorem or by using the remark above, but we could not give a precise proof. Under this assumption we can interpret equation \eqref{firstdef} as an ordinary differential equation for a function $u(d)$  that depends only on the distance $d=d(x,\partial\Omega)$ to the boundary, with initial condition \eqref{bc1} and \eqref{bc2} at $d=0$.  
Then we arrive after the first integration at $u_\nu^3(d)-a^3=-3d$ or $-u_\nu=(a^3-3d)^{1/3}$ and after a second integration at
$$u(d)=\int_0^d (a^3-3t)^{1/3} \ dt=\frac{1}{4}\left[a^4-\left(a^3-3d\right)^{4/3}\right].$$
Clearly the integrations are only justifiable for sufficiently small $d$ and as long as $d$ is locally of class $C^{1,1}$. When $d=a^3/3$, the gradient of $u$ vanishes and we have reached the peak on our way uphill from the boundary. This shows that $\Omega$ has an inradius of exactly $a^3/3$. Incidentally, the points in $$M(\Omega):=\{y\in\Omega\ |\ d(y,\partial\Omega)=\max_{x\in\Omega} d(x,\partial\Omega)\}$$ belong to the {\it ridge} of $\Omega$ or {\it cut locus }Êof $\partial\Omega$, which is defined as follows.
Let $G$ be the largest open subset of $\Omega$ such that every point $x$ in $G$ has a unique closest point on ${\partial\Omega}$. Then we call
$${\cal R}(\Omega):=\Omega\setminus G$$
the ridge ${\cal R}(\Omega)$. In $G$, the distance $d(x,\partial\Omega)$ to the boundary is at least of class $C^1$, and also smooth, i.e., of class $C^{2}$ or $C^{k,\alpha}$ with $k\geq 2$ and $\alpha\in(0,1)$ provided $\partial\Omega$ is of the same class, see \cite{CM, LN2}. It is remarkable that even for a convex plane domain the ridge can have positive measure, see pages 10 and 11 in \cite{MM}. Simple examples such as an ellipse or a rectangle show that in general $M(\Omega)$ is a genuine subset of the ridge, but there are many domains with the property $M(\Omega)={\cal R}(\Omega)$. 

\medskip

Examples of such domains are for instance a stadium domain (convex hull of two balls of same radius and different center), an annulus, or plane domains which are generated as follows. Let $\gamma$ be a compact $C^{1,1}$ curve with curvature not exceeding $K$ in modulus and $\Omega=U_b(\gamma)=\{ x\in\R^2\ |\ d(x,\gamma)<b\}$ with $b<1/K$. Then $M(\Omega)={\cal R}(\Omega)$, see Figure 1.

\begin{figure}[h]
\centerline{\includegraphics[height=5truecm]{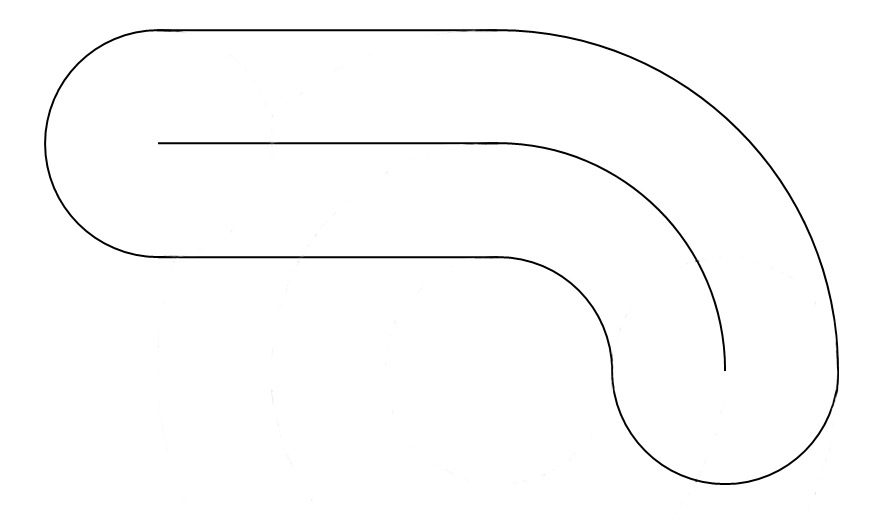}} 
\caption{A domain satisfying $M(\Omega)={\cal R}(\Omega)$}
\end{figure}

\begin{thm}\label{firsthm}
Suppose that $\partial\Omega$ is of class $C^{2}$. Then  a webfunction $u\in C^1(\overline{\Omega})$ is a viscosity solution of \eqref{firstdef} \eqref{bc1} \eqref{bc2} if and only if $M(\Omega)={\cal R}(\Omega)$ and every $x\in\partial\Omega$ has distance $a^3/3$ to ${\cal R}(\Omega)$.
\end{thm}

\begin{proof} In fact, if $M(\Omega)={\cal R}(\Omega)$, then the function
$$u(x)= \frac{1}{4}\left[a^4-\left(a^3-3d(x,\partial\Omega)\right)^{4/3}\right]$$
is well defined and differentiable everywhere in ${\Omega}$. Moreover, according to \cite{CM}, it is of class $C^2(\Omega\setminus{\cal R}(\Omega))$ and solves \eqref{firstdef} in $\Omega\setminus{\cal R}(\Omega)$ in the classical (and a fortiori in the viscosity) sense. Finally on $M(\Omega)={\cal R}(\Omega)$ we can argue as in the radial case to see that $u$ is a viscosity solution there as well. This shows that the geometric constraint $M(\Omega)={\cal R}(\Omega)$ is sufficient for the existence of solutions to \eqref{firstdef} \eqref{bc1} \eqref{bc2}.

To prove necessity, suppose that $M(\Omega)$ is a genuine subset of ${\cal R}(\Omega)$, so that there exists a $z\in {\cal R}(\Omega)\setminus M(\Omega)$. But then $d(z,\partial\Omega)< {a^3}/{3}$ and $d(z,\partial\Omega)$ has a kink in the sense that some directional derivative of $d$, and subsequently of $u$, is discontinuous at $z$. This is incompatible with being a viscosity solution, because one can then find an admissible test function $\varphi\in C^2(\Omega)$ for which $F(D\varphi, D^2\varphi)$ fails to satisfy the proper inequality. To be precise, suppose that $\Omega$ is essentially a rectangle (with rounded corners to make it smooth) or an ellipse. Then $z$ lies on a line segment and $d(x,\partial\Omega)$ is concave near $z$ and has one-sided nonzero derivatives in direction $\eta$ orthogonal to the ridge in $z$. But then one can choose a $C^2$ function $\varphi$, touching $u$ from above in $z$ such that $\nabla\varphi(z)\not=0$ points in direction $\eta$ and $\varphi_{\eta\eta}(z)<-K$, where $K$ is an arbitrarily large number. Therefore $F(D\varphi,D^2\varphi)(z)>0$, which contradicts the requirement for subsolutions. There is a similar reasoning using supersolutions, if $\Omega$ is essentially $L$-shaped and $u$ is convex and nondifferentiable on parts of its ridge.
\end{proof}

\section{The normalized or game-theoretic $\infty$-Laplacian }\label{section3}\setcounter{equation}{0}

Recently the following operator has received considerable attention (see for instance \cite{PSSW,PS,JK,LW,LW2,Y2}) in the PDE community
$$\Delta_\infty^Nu=\langle D^2u Du, Du\rangle|Du|^{-2}.$$
Here $u(x)$ denotes the (unique) running costs in a differential game called \lq\lq tug of war\rq\rq, see \cite{Y2}. 
Let us therefore study the differential equation
\begin{equation}\label{seconddef}
-u_{\nu\nu}=1\qquad\hbox{in }\Omega
\end{equation}
under boundary conditions \eqref{bc1} and \eqref{bc2}. A simple integration shows that certainly for a ball of radius $R=a$ this overdetermined problem has the explicit solution $u(r)=(a^2-r^2)/2$, provided we can live with the ambiguity that $\nu$ is not properly defined at the origin. Fortunately the notion of viscosity solution allows us to do so. A viscosity solution $u$ of 
\begin{equation}\label{Geqn}
G(Du,D^2u)(x):=-\frac{\langle D^2uDu,Du\rangle}{|Du|^2}(x)-1=0\qquad\hbox{in }\Omega
\end{equation}
is a viscosity subsolution of $G_*(Du,D^2u)=0$ and a viscosity supersolution of $G^*(Du,D^2u)=0$. Here $G_*$ and $G^*$ are the upper and lower semicontinuous envelopes of $G$, see Remark 6.3 in \cite {CIL}. Thus $u\in C(\Omega)$ is a {\it viscosity subsolution} of \eqref{seconddef} or \eqref{Geqn}, if for every $x\in\Omega$ and every smooth test function $\varphi$, that touches $u$ from above (only) in $x$, the following relations hold:
\begin{equation}\label{Gsubsol}
\begin{cases}
G(\nabla\varphi(x),D^2\varphi(x))\le0 &\hbox{ when }\nabla\varphi(x)\not=0, \\
\quad-\Lambda(D^2\varphi(x))-1\le0 &\hbox{ when }\nabla\varphi(x)=0.
\end{cases}
\end{equation}
In a similar fashion {\it viscosity supersolutions} $u\in C(\Omega)$ of \eqref{seconddef} are characterized by the fact that
\begin{equation}\label{Gsupersol}
\begin{cases}
G(\nabla\psi(x),D^2\psi(x))\geq 0 &\hbox{ when }\nabla\psi(x)\not=0, \\
\quad-\lambda(D^2\psi(x))-1\geq 0 & \hbox{ when }\nabla\psi(x)=0.
\end{cases}
\end{equation}
for every smooth test function $\psi$ that touches $u$ from below (only) in $x$. Here $\Lambda(X)$ and $\lambda(X)$ denote the maximal and minimal (nonnegative) eigenvalue of the symmetric matrix $X$. 

For a more general $\Omega$, if we interpret \eqref{seconddef} again as an ODE and \eqref{bc1} and \eqref{bc2} as initial data on $\partial\Omega$, then an integration like in the previous section along lines of steepest ascent of $u$ leads to the local representation
$$u(x)=\frac{d(x,\partial\Omega)}{2}\,\big(2a-d(x,\partial\Omega)\big)\quad\hbox{in }\Omega\setminus{\cal R}(\Omega).$$

\begin{thm}
Suppose that $\partial\Omega$ is of class $C^{2}$. Then a webfunction $u\in C^{1}(\overline{\Omega}) $ is a viscosity solution of \eqref{seconddef} \eqref{bc1} \eqref{bc2} if and only if $M(\Omega)={\cal R}(\Omega)$ and every $x\in\partial\Omega$ has distance $a$ to ${\cal R}(\Omega)$. 
\end{thm}

The proof parallels the one of Theorem \ref{firsthm} and is left to the reader. 

\begin{Remark}{\rm
Notice that annuli provide examples of domains (other than balls) for which a smooth solution of this problem (but not of Serrin's and Weinberger's original problem) exists.}
\end{Remark}

\section{The limit of $u_p$ }\label{section4}\setcounter{equation}{0}

It is well-known, that $p$-harmonic functions or viscosity solutions of $\Delta_p u=0$ converge to the viscosity solution of $\Delta_\infty u=0$ as $p\to\infty$. Therefore one is inclined to believe that solutions $u_p$ of the inhomogeneous equation \eqref{op} should converge to those of \eqref{firstdef}. This is not the case, and in the present section we investigate this limit. For $\Omega$ a ball in $\R^n$ the solutions of \eqref{op}, \eqref{bc1} were explicitly calculated and shown to converge uniformly to $d(x,\partial\Omega)$ in \cite{K}. Let us demonstrate that this behaviour happens for any connected domain, even for a nonsmooth one. First one has to note that $u_p$ on $\Omega$ can be estimated in $L^q$ for any $q\in[0,\infty]$ by the corresponding solution $U_p$ on a ball $\Omega^*$ of same volume as $\Omega$, so that the $u_p$ are uniformly bounded in $L^\infty(\Omega)$ as $p\to \infty$. Furthermore $u_p$ minimizes the functional
$$J_p(v)=\int_\Omega\Big[\frac{1}{p}|\nabla v(x)|^p-v(x)\Big]\,dx\qquad\hbox{on }W^{1,p}_0(\Omega).$$
In particular
$$J_p(u_p(x))\le J_p(d(x,\partial\Omega))=\frac{1}{p}|\Omega|-\int_\Omega d(x,\partial\Omega)\,dx,$$
the right hand of which is negative for sufficiently large $p$. Thus
$$\int_\Omega|\nabla u_p|^p\,dx\le p\int_\Omega u_p\,dx,$$
or for $p>q$ and $q$ large enough
$$\int_\Omega|\nabla u_p|^q\,dx\le\left(\int_\Omega|\nabla u_p|^p\,dx\right)^{q/p}|\Omega|^{1-q/p}\le\left(p\int_\Omega u_p\,dx\right)^{q/p}|\Omega|^{1-q/p}.$$
But this implies $||\nabla u_p||_q\leq p^{1/p}||u_p||^{1/p}_\infty |\Omega|^{1/q}$, so that the family $\left\{ u_p\right\}_{p\to\infty}$ is uniformly bounded in every $W^{1,q}(\Omega)$ and converges uniformly to some limit $u_\infty$ with Lipschitz constant $1$.

Therefore $|\nabla u_\infty| \leq 1$ a.e. in $\Omega$, and this implies not only that $u_\infty(x)\leq d(x,\partial\Omega)$ in $\Omega$, but it (almost) proves the first half of our following result.

\begin{thm}\label{thirdthm}
The limit $u_\infty$ is a viscosity solution of the eikonal equation $|Du(x)|-1=0$ in $\Omega$ under the Dirichlet boundary condition $u=0$ on $\partial\Omega$.
\end{thm} 

\begin{Remark}{\rm
Since this Hamilton-Jacobi equation has a unique viscosity solution, see e.g. \cite{CIL}, we obtain $u_\infty:=d(x,\partial\Omega)$ as a Lipschitz solution for a highly overdetermined boundary value problem. It satisfies not only $|Du|-1=0$ in $\Omega$ but also $-\Delta_\infty u_\infty=0$ in $\Omega\setminus{\cal R}(\Omega)$, and not only $u=0$ on $\partial\Omega$ but also $-\tfrac{\partial u}{\partial\nu}=1$ on differentiable parts of $\partial\Omega$.}
\end{Remark}

\begin{Remark}{\rm
Notice that the statement $M(\Omega)={\cal R}(\Omega)$ is conspicuously missing in Theorem \ref{thirdthm}. Under the additional assumption $M(\Omega)={\cal R}(\Omega)$, however, the function $u_\infty$ is moreover (up to multiplication by a constant) the unique eigenfunction for the $\infty$--Laplacian operator, i.e. it satisfies in addition
$$\min\left\{-\langle D^2u_\infty(x) Du_\infty(x),Du_\infty(x)\rangle, \ -|Du(x)|+\Lambda_\infty u(x)\right\}=0\qquad\hbox {in }\Omega$$
in the viscosity sense, see \cite{J, Y}. Here $\Lambda_\infty$ is the inverse of the inradius of $\Omega$. Without this assumption, as demonstrated in \cite{JLM} there is nonuniqueness of this eigenfunction.}
\end{Remark}

{\it Proof of Theorem \ref{thirdthm}.}
Let us first realize that $|Du_\infty|\leq 1$ a.e. in $\Omega$ implies $|Du_\infty|-1\leq 0$ in the viscosity sense. Otherwise there would be a function $\varphi\in C^2$ touching $u$ from above in some $x_0$ such that $|Du(x_0)|\geq 1+\gamma$, with $\gamma >0$, and $|Du(x)\geq1+\gamma/2$ in a neighbourhood  $B_\varepsilon(x_0)$. But then $u(x_0)-u(x)\geq \varphi(x_0)-\varphi(x)\geq (1+\gamma/2)|x_0-x|$ for a suitable $x\in B_\varepsilon(x_0)$. This contradicts the fact that $u_\infty$ has Lipschitz constant 1.

To show the reverse inequality, it is instructive to follow ideas in \cite{J, BDM} and to identify the limiting equation. Suppose that $\varphi$ is a $C^2$-function such that $\varphi-u_\infty$ has a local minimum at $x_0\in \Omega$. Then without loss of generality we may assume that $\varphi-u_\infty\geq \delta>0$ on $\partial B_\varepsilon(x_0)\subset\Omega$. Moreover, for $p$ large enough,  $\varphi-u_p$ has a local minimum at some $x_p\in B_\varepsilon(x_0)$ and $x_p\to x_0$ as $p\to \infty$. Since $u_p$ is a viscosity subsolution of (\ref{op})
\begin{equation}\label{opv}
- |Du|^{p-2}\left( tr(D^2u) +(p-2)\frac{\langle D^2u Du, Du\rangle }{|Du|^2}\right) -1 =0\quad\hbox{ in } \Omega,
\end{equation}
it follows
$$- |D\varphi(x_p)|^{p-2}\left( tr(D^2\varphi(x_p)) +(p-2)\frac{\langle D^2\varphi(x_p) D\varphi(x_p), D\varphi(x_p)\rangle }{|D\varphi(x_p)|^2}\right)\leq 1.$$
Now either $|D\varphi(x_0)|\leq 1$ or otherwise there exists a positive constant $\gamma$ independent of $p$, such that $|D\varphi(x_p)|>1+\gamma$ for large $p$.
Upon division of the last inequality by $(p-2)|D\varphi(x_p)|^{p-4}$ one sees that in this case the first term on the left and the right hand side in 
$$-\tfrac{1}{p-2}|D\varphi(x_p)|^2 tr D^2\varphi(x_p)-\langle D^2\varphi(x_p) D\varphi(x_p), D\varphi(x_p)\rangle\leq \tfrac{1}{p-2}|D\varphi(x_p)|^{4-p}$$
converge to zero as $p\to\infty$, so that 
$-\langle D^2\varphi(x_0) D\varphi(x_0), D\varphi(x_0)\rangle\leq 0$. This proves that  $u_\infty$ is a viscosity subsolution of 
\begin{equation} \label{limeq}
\min \{\ |Du|-1, -\langle D^2u Du, Du\rangle\ \}= 0\quad\hbox{ in } \Omega.
\end{equation}
A similar reasoning holds for supersolutions. Since $u_p$ is a viscosity supersolution of (\ref{opv}), we have
$$- |D\psi(x_p)|^{p-2}\left( tr(D^2\psi(x_p)) +(p-2)\frac{\langle D^2\psi(x_p) D\psi(x_p), D\psi(x_p)\rangle }{|D\psi(x_p)|^2}\right)\geq 1$$
for testfunctions $\psi\in C^2$ such that $u-\psi$ has a local maximum at $x_0$ and $u_p-\psi$ has a local maximum at $x_p$. This time we can rule out that $D\psi(x_p)=0$, otherwise the last inequality cannot hold. Arguing as before, the inequality
$$-\tfrac{1}{p-2}|D\psi(x_p)|^2 tr D^2\psi(x_p)-\langle D^2\psi(x_p) D\psi(x_p), D\psi(x_p)\rangle\geq \tfrac{1}{p-2}|D\psi(x_p)|^{4-p}$$
follows and leads to $|D\psi(x_0)|\geq 1$, because else the right hand side would explode for $p\to\infty$, as well as  to
$-\langle D^2\psi(x_0) D\psi(x_0), D\psi(x_0)\rangle \geq 0$. This shows that $u_\infty$ is also a viscosity supersolution of (\ref{limeq}). In particular $u_\infty$ satisfies $|Du|\geq 1$ in the viscosity sense, and this completes the proof of Theorem \ref{thirdthm}.\hfill\qed

\bigskip
\noindent{\bf Acknowledgement:} This paper was conceived in December 2008 during a GNAMPA-INDAM visit of B. Kawohl to Pisa. This research is also part of the ESF program {\it``Global and geometrical aspects of nonlinear partial differential equations''}. The authors are gratefully indebted to the referee for helpful questions that led to an improvement of the paper.

\bigskip
\noindent
Author's addresses:

\bigskip
{\small
\begin{minipage}[t]{6.3cm}
Giuseppe Buttazzo\\
Dipartimento di Matematica\\
Universit\`a di Pisa\\
Largo B. Pontecorvo, 5\\
I-56127 Pisa - ITALY\\
{\tt buttazzo@dm.unipi.it}
\end{minipage}
\begin{minipage}[t]{6.3cm}
Bernd Kawohl\\
Mathematisches Institut\\
Universit\"at zu K\"oln\\
Albertus-Magnus-Platz\\
D-50923 K\"oln - GERMANY\\
{\tt kawohl@mi.uni-koeln.de}
\end{minipage}}

\end{document}